\documentclass[11pt,a4paper]{amsart}

\usepackage{lmodern}

\usepackage[english]{babel}

\usepackage[utf8]{inputenc} 	% allow utf-8 input
\usepackage[T1]{fontenc}    	% use 8-bit T1 fonts
\usepackage{hyperref}       	% hyperlinks
\usepackage{amssymb}            
\usepackage{url}            		% simple URL typesetting
\usepackage{booktabs}       	% professional-quality tables
\usepackage{amsfonts}       	% blackboard math symbols
\usepackage{centernot}

\usepackage{nicefrac}       	% compact symbols for 1/2, etc.
\usepackage{microtype}      	% microtypography
\usepackage{lipsum}			% Can be removed after putting your text content
\usepackage{graphicx}
\usepackage{subcaption}
\usepackage{hyperref} 				% For hyperlinks in the PDF
\usepackage{tikz}
\usepackage{amsmath}
\usepackage{enumerate}
\usepackage{mwe}
\usepackage[shortlabels]{enumitem} 	% Customized lists
\setlist[itemize]{noitemsep} 			% Make itemize lists more compact
\usepackage{graphbox}				% centerr align figures
\usepackage{float}
\usepackage[toc,page]{appendix}
\usepackage{mathtools}
\usepackage{stmaryrd}
\usepackage{faktor}
\usepackage{cancel}
\usepackage{tikz-cd}
\usepackage[alphabetic, abbrev]{amsrefs} 
\usepackage{tkz-euclide}

\usepackage[normalem]{ulem}
\usepackage{pifont}

\usepackage{tikz}
\usepackage{stackengine}

\usepackage{upgreek}

\theoremstyle{plain}
\newtheorem*{definition*}{Definition}
\newtheorem{theorem}{Theorem}%[section]

\newtheorem{proposition}[theorem]{Proposition}
\newtheorem{lemma}[theorem]{Lemma}

\theoremstyle{definition}

\theoremstyle{remark}
\newtheorem*{example*}{Example}
\newtheorem{remark}[theorem]{Remark}
\newtheorem{example}[theorem]{Example}

\hypersetup{
    colorlinks,
    linkcolor={red!50!black},
    citecolor={blue!50!black},
    urlcolor={blue!80!black}
}

\newcommand{\R}{\mathbb{R}}				% R
\newcommand{\N}{\mathbb{N}}				% N
\newcommand{\F}{\mathcal{F}}				% Family of vector fields
\newcommand{\A}{\mathcal{A}}				% Attainable set
\newcommand{\U}{\mathcal{U}}				% Control set

\newcommand{\spa}{\mathrm{span}}

\newcommand{\cl}{\operatorname{cl}}

\newcommand{\interior}{\operatorname{Int}}
\providecommand{\keyword}[1]
{
  \small	
  \textbf{Keywords:} #1
}

\usepackage[bottom=3cm, top=3cm, left=3.5cm, right=3.5cm]{geometry}

\tikzset{degil/.style={
            decoration={markings,
            mark= at position 0.5 with {
                  \node[transform shape] (tempnode) {$\backslash$};
                  }
              },
              postaction={decorate}
}
}

\title[Local controllability implies controllability]{Local controllability does imply \\global controllability}

\date{\today}

\author{Ugo Boscain}
\address{Ugo Boscain, CNRS, Laboratoire Jacques-Louis Lions, Sorbonne Universit\'e, Universit\'e de Paris, Inria,  Boîte courrier 187, 75252 Paris Cedex 05 Paris, France}
\email{ugo.boscain@upmc.fr}

\author{Daniele Cannarsa}
\address{Daniele Cannarsa, Department of Mathematics and Statistics, University of Jyv\"askyl\"a, 40014 Jyv\"askyl\"a, Finland
}
\email{daniele.d.cannarsa@jyu.fi}

\author{Valentina Franceschi}
\address{Valentina Franceschi, Dipartimento di Matematica Tullio Levi-Civita, Università di Padova, Italy}
\email{valentina.franceschi@unipd.it}

\author{Mario Sigalotti}
\address{Mario Sigalotti, 
Sorbonne Universit\'e, Inria, CNRS, Laboratoire Jacques-Louis Lions (LJLL), F-75005 Paris, France
}
\email{mario.sigalotti@inria.fr}

\newcommand{\STLC}{\hyperlink{eq:STLC}{ST}}
\newcommand{\LLC}{\hyperlink{eq:LLC}{L}}
\newcommand{\STLLC}{\hyperlink{eq:STLLC}{STL}}

\begin{document}

\maketitle

\begin{abstract}
We say that a control system is locally controllable if the attainable set from any state $x$ contains an open neighborhood of $x$, while it is controllable if the attainable set from any state is the entire state manifold.
%{\color{red}Quite surprisingly, the question of whether local controllability implies global controllability seems not to have been considered in the literature.}
We show in this note that a control system satisfying  local controllability is  controllable. 
Our self-contained proof is alternative to the combination of two previous results by Kevin Grasse.
\end{abstract}
\bigskip\bigskip
{ \keyword{local controllability, small-time local controllability, controllability}}
 \bigskip
\setcounter{tocdepth}{1}

\section{Introduction}

Let $M$ be a connected finite-dimensional 
smooth manifold. %, with $n\geq 1$.
Here we study the controllability properties of a 
%control 
system of the form 
	\begin{equation}\label{eqn:system}\tag{C}
	\dot x = F(x, u(t)), \qquad x\in M, %~ F(\cdot,u)\in \F,
	\end{equation}
	where $F:M\times U\to TM$, with $U\subset\R^{m}$ for some $m\in \N$ and $F$ smooth. 
We consider as set of admissible controls either 
 $\U_{pc}=\cup_{T\ge0}\{u:[0,T]\to U\mid \mbox{$u$ piecewise constant}\}$, or the family $\mathcal{U}_{\infty}=\cup_{T\ge0}L^\infty([0,T],U)$.
Given an admissible control $u\in \mathcal{U}$, where $\U$ is one of these two classes, we denote by $\phi(\cdot,x,u)$ the unique absolutely continuous maximal solution of 
\eqref{eqn:system} with initial condition $x$ at time $0$. %Check reference
The \textit{attainable set} from a point $x$ in $M$ for system \eqref{eqn:system} is
	\begin{equation}\label{eq:attainable}
	    \A_{x}	= \{\phi(T,x,u)\mid T\geq 0, ~ u\in\mathcal{U},\;\phi(\cdot,x,u)\  \mbox{is defined on }[0,T]\}.
	\end{equation}

	When $y\in \A_{x}$ we say that $y$ is \emph{attainable} or 
\emph{reachable from $x$}.
We say that system \eqref{eqn:system}  is \textit{locally controllable} if  the attainable set from any initial state $x$ in $M$ is a  neighborhood of $x$, i.e.,
	\begin{equation} %\label{eq:LC}
	%\tag{LC}
	{\color{black} x\in \interior \A_{x},\qquad \forall x \in M,}
	\end{equation}
while it %System \eqref{eqn:system} 
is said to be  \emph{controllable} if $\A_{x}=M$ for each $x$ in $M$. 
Observe that sometimes in the literature the expression   \emph{local controllability} is used with a different meaning (see, for example,  \cite[Definition~3.2]{Coron}).

The notion of local controllability has been studied extensively in the literature, especially in the stronger forms of \emph{small-time local controllability} 
(\STLC-local controllability) (for which the attainable set $\A_x$ is replaced by the set of points attainable from $x$ within an arbitrarily small positive time) 
 and \emph{localized local controllability}  (\LLC-local controllability)
(for which one considers the set of points attainable from $x$ by admissible trajectories that stay in an arbitrarily small neighborhood of $x$). 
A combination of the two constraints yields the notion of \emph{small-time localized local controllability} (\STLLC-local controllability).
Table~\ref{tab:schema} contains a scheme of the implications that can be directly deduced from the above definitions, and we refer to Section~\ref{s:discussion} for a detailed description of these different types of local controllability and the relations between them.

\begin{table}
{\small
\begin{tikzcd}[column sep=0pt]
& \text{\STLC-{local controllability}}   \arrow[rd,Rightarrow] && \\
\text{\STLLC-{local controllability}} \arrow[ru,Rightarrow] \arrow[rd,Rightarrow]&
&
\text{{local controllability}}
 \arrow[r,bend right,Rightarrow,"\text{Thm.~\ref{thm:thm-1}}"']  &\mbox{\quad controllability} \arrow[l,bend right,Rightarrow]  \\
& \text{\LLC-{local controllability}}  \arrow[ur,Rightarrow] &
\end{tikzcd}
}
\caption{\label{tab:schema} Relations between different types of local controllability. As discusses in Section~\ref{s:discussion}, the missing arrows cannot be added to the scheme.
The only arrow that needs to be justified here is the one representing the fact that local controllability implies controllability. This is the object of Theorem~\ref{thm:thm-1}.}
\end{table}

Folklore has it that controllability can be deduced from 
suitable 
%the stronger 
versions of local controllability: for example, in \cite[Section~12.3]{parsons2007} it is stated ({without proof}) that \STLLC-local controllability implies controllability.
Another example are linear systems for which controllability is known to be equivalent to 
\STLC-local controllability.
%{\color{red}
%The question whether {\STLC-local controllability implies controllability} for a general control system {remained} an open question (see, e.g., \cite[Sec.~3]{MR689492}).
%}
The question whether {\STLC-local controllability implies controllability} for a more general control system was formulated for instance in \cite[Sec.~3]{MR689492}.

The purpose of this paper is to prove that the weakest version of local controllability is sufficient to deduce controllability,  {giving in particular a positive answer to the question just mentioned.}

\begin{theorem}\label{thm:thm-1}
	If system \eqref{eqn:system} is locally controllable, then  \eqref{eqn:system} is controllable.
\end{theorem}

We mention that
if \LLC-local controllability is known, then controllability can be shown with a simpler proof than what is proposed here for Theorem~\ref{thm:thm-1} {(see Appendix \ref{app})} but still does not follow immediately from the definitions, since reachability is not a symmetric property.
Indeed, the fact that one can reach an open neighborhood of a given initial state does not imply %directly 
that any point in the neighborhood can be steered %controlled 
back to the initial state. %steered

Let us mention that it is hard to find testable conditions for local controllability to hold. Indeed, in the literature it is more common to find conditions for \STLC-local controllability since those can be deduced from Lie algebraic arguments (see, e.g., \cite{Krastanov2021} and references therein).
%one can utilise the Lie algebra generated by $\F$.
It should be noticed that such conditions do not usually provide 
local controllability at \emph{every} point, since they 
tipically
require that the point at which local controllability is studied is an equilibrium of one of the admissible vector fields. 
Finally, we note that the interest in \STLC-local controllability is motivated, for example, by its relation with the continuity of the optimal time function, as explained in \cite{MR872457}.

\subsection*{Key steps of the proof of Theorem \ref{thm:thm-1}}
System \eqref{eqn:system} is said to be \textit{approximately controllable} if all attainable sets are dense, i.e., 
\begin{equation*}%\label{eq:AC}\tag{AC}
\cl\A_x=M, \qquad \forall~ x\in M,
\end{equation*}
where $\cl$ denotes the closure with respect to the topology of $M$.
The first step in the proof of of Theorem \ref{thm:thm-1} is to prove that local controllability implies approximate controllability.
\begin{lemma}\label{lem:approx}
If \eqref{eqn:system} is locally controllable, then it is approximately controllable.
\end{lemma}

The proof of the lemma, presented in Section~\ref{sec:proofs}, relies on the regularity of the flow of \eqref{eqn:system} for a fixed control, and on the connectedness of $M$. 
The second key property is the following lemma. 
\begin{lemma}\label{lem:salmone}
 Assume that \eqref{eqn:system} is locally controllable. Then, for every $x$ and $y$ in $M$, \[ y\in \A_{x}\implies x\in \A_{y}.\]
\end{lemma}

Lemma~\ref{lem:salmone} is proved by showing that trajectories of \eqref{eqn:system} can be retraced back by finding a control driving their endpoint to their starting point. 
More precisely, assume that $y$ is in $\A_{x}$ and consider a control $u$ such that $y=\phi(T,x,u)$. 
For $t$ in a left neighborhood of $T$, the point $\phi(t,x,u)$ can be reached from $y$ due to local controllability. 
By repeating this argument and concatenating the controls, one can find smaller and smaller $t\geq0$ such that $\phi(t,x,u)$ can be reached from $y$.
In order to reach $x=\phi(0,x,u)$ (i.e., to prove the lemma) one has to show that the sequence of times $t$ found following such a procedure eventually attains zero, unlike the situation depicted in  Figure~\ref{img:risali}.
\\
 
\begin{figure}\centering
\begin{tikzpicture}
    \draw (0, -0.05) node[inner sep=0] {\includegraphics[width=0.65\linewidth]{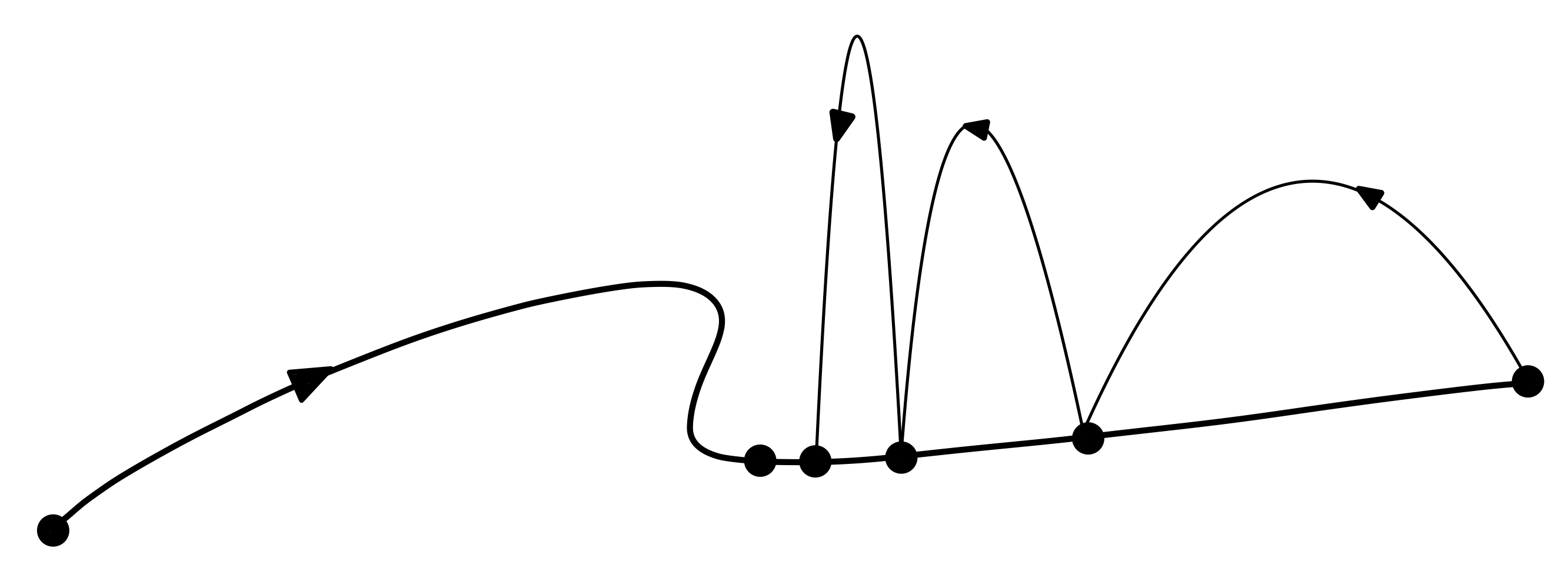}};
  \draw (-2, 0.3) node {$\phi(\cdot,x,u)$};
    \draw (0.1, -1.4) node {$z \phantom{\cdots}$};
    \draw (4.6, -0.9) node {$y$};
    \draw (-4.5, -1.7) node {$x$};
\end{tikzpicture} 
\caption{\label{img:risali}When retracing back the trajectory $\phi(\cdot,p,u)$, the attainable sets might get smaller and smaller and collapse to a point $z$ before reaching $x$, since a priori their size is not lower semi-continuous.  Lemma~\ref{lem:salmone} shows that this situation cannot happen, proving a key step for the proof of Theorem \ref{thm:thm-1}.}
\end{figure}
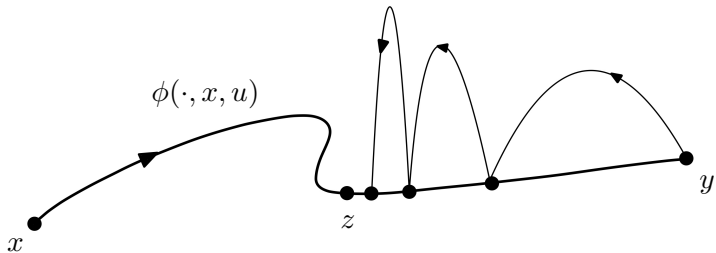

These arguments for the proof of Theorem \ref{thm:thm-1} hold for other classes of controls, provided that the control system remains well-posed in the space of absolutely continuous functions and that the set of controls contains the piecewise constant functions. 
{(For conditions of this type, see, e.g., \cite[Chapters 2 and 3]{bressan}.)}
\\

The fact that local controllability implies global controllability was already treated in \cite{MR692840} for a compact state manifold $M$. 
For the noncompact case, in \cite{MR774166} the author shows that local controllability implies global one for piecewise constant controls.
A related result in \cite{MR1346883} shows that  local controllability via bounded measurable controls implies the same via piecewise constant controls, therefore extending the previous result to control systems with bounded measurable controls.  We are grateful to Kevin Grasse for attracting our attention on these papers after a first version of our note was presented.
%Our self-contained proof is alternative to the combination of these two results.

\subsection*{Acknowledgements}
This work was supported by the Grant %ANR-15-CE40-0018 SRGI 
{ANR-17-CE40-0007-01 QUACO}
of the French ANR.

% %%%%%%%%%%%%%%%%%%%%%%%%%%%%%%%%%%%%

\section{On the local forms of controllability} \label{s:discussion}

Let us define the set of points attainable from a point $x$ in $M$ at a time $t>0$ with trajectories of \eqref{eqn:system} remaining inside of a domain $\Omega$ by
\[
	\A_{x,\Omega}^{t}= \{\phi(t,x,u)\mid  u\in\mathcal{U},\;\phi(\cdot,x,u)\  \mbox{defined on }[0,t]\text{ with values in }\Omega
	\},
\]
and let 
\[
	\A_{x,\Omega}^{\leq T}= \bigcup_{0<t\leq T}\A_{x,\Omega}^{t},\qquad \A_{x,\Omega}= \bigcup_{0<t<+\infty}\A_{x,\Omega}^{t}.
\]
Moreover, let us denote $\A_{x}^{\le T}=\A_{x,M}^{\leq T}$ and notice that $\A_{x}=\A_{x,M}^{\leq +\infty}$.
System \eqref{eqn:system} is said to be small-time locally controllable if
\begin{equation}
%\label{eq:STLC}\tag{ST}
\tag{\hypertarget{eq:STLC}{ST}-locally controllable}
x\in\interior\A_x^{\leq T} \qquad \forall\;T>0,\;\forall\; x\in M;
\end{equation}
moreover, system  \eqref{eqn:system} is said to be localized  locally controllable  if
\begin{equation}%\label{eq:LLC}\tag{L}
\tag{\hypertarget{eq:LLC}{L}-locally controllable}
x\in\interior\A_{x,\Omega} \qquad \forall\; x\in M,\;\forall\;\Omega \mbox{ neigh.~of } x.
\end{equation}
Finally, system \eqref{eqn:system} is said to be small-time localized locally controllable if 
\begin{equation}%\label{eq:STLLC}\tag{STL}
\tag{\hypertarget{eq:STLLC}{STL}-locally controllable}
x\in\interior\A_{x,\Omega}^T \qquad \forall\;T>0,\; \forall \;x\in M,\;\forall\;\Omega \mbox{ neigh.~of } x.
\end{equation}
One recognizes immediately that the implications contained in Table~\ref{tab:schema}  can be directly deduced from the above definitions, with the exception of Theorem \ref{thm:thm-1}. 
In this section we show that the missing arrows cannot be added to the scheme, provinding the examples summarised in Table \ref{tb:non-implicazioni}.

\begin{table}
\begin{tikzcd}[column sep=0pt
]
& \text{\STLC-{local controllability}}  
 \arrow[rd,Rightarrow,shift right =2,shorten >=3ex] 
\arrow[rd,Leftarrow,degil,shift left=2,shorten <=3ex,"\text{~~Ex.~\ref{rk:LnonST}}"] 
 \arrow[dd,Rightarrow,degil,shift left =2,"\text{~~Ex.~\ref{r:linear-LLC}}"] \arrow[dd,Leftarrow,degil,shift right=2,"\text{~~Ex.~\ref{rk:LnonST}}~~"']
& \\
\text{\STLLC-{local controllability}} \arrow[ru,Rightarrow,shift right =2,shorten <=3ex,] \arrow[ru,Leftarrow,degil,shift left=2,"\text{~~Ex.~\ref{r:linear-LLC}}~",shorten >=3ex,] \arrow[rd,Rightarrow,shift left =2,shorten <=3ex,]\arrow[rd,Leftarrow,degil,shift right=2,"\text{~~Ex.~\ref{rk:LnonST}}"',shorten >=3ex,] &
&
\text{{local controllability}}
 \\
& \text{\LLC-{local controllability}}  \arrow[ur,Rightarrow,shift left =2,shorten >=3ex] \arrow[ru,Leftarrow,degil,shift right=2,"\text{~~Ex.~\ref{r:linear-LLC}}"',shorten <=3ex]&
\end{tikzcd}

\caption{\label{tb:non-implicazioni}A diagram summarizing the fact that the weaker forms of local controllability do not imply the stronger.}
\end{table}

\begin{example}[\STLC-local controllability $\centernot\implies$\LLC-local controllability]\label{r:linear-LLC}
Recall that a linear control system $\dot x=Ax+Bu$ is controllable if and only if it is 
\STLC-locally controllable which, in turns, is equivalent to the
Kalman condition (see, e.g, \cite{MR1640001}). 
On the other hand, a control system is \LLC-locally controllable only if the range ${\rm Im}(B)$ of $B$  is the entire state space.
Indeed, 
if  $A\bar x$ points outside  ${\rm Im}(B)$ at some $\bar x\in  {\rm Im}(B)$, considering a linear system of coordinates associated with a basis containing $A\bar x$ and a set of generators for ${\rm Im}(B)$, we have that the component along $A\bar x$ of  every admissible trajectory staying in a sufficiently small neighborhood of $\bar x$ is increasing. Hence, a necessary condition for \LLC-local controllability is that $Ax$ is in ${\rm Im}(B)$ for every $x\in {\rm Im}(B)$. If such a condition is satisfied, every admissible trajectory starting from ${\rm Im}(B)$ cannot exit it. Therefore, 
 \LLC-local controllability can only hold 
when ${\rm Im}(B)$ is maximal. % Mettere sezione finale
(See \cite{MR689492} for more results on controllability and local controllability of control-affine systems with unbounded controls.)

Hence any controllable linear system such that ${\rm Im}(B)$ is not maximal
is \STLC-locally controllable without being  \LLC-locally controllable.
This also proves that %a \STLC-locally controllable system is not necessarily locally controllable and that
a \STLC-locally controllable system is not necessarily  \STLLC-locally controllable,
and that a locally controllable system is not necessary \LLC-locally controllable,
as indicated in Table~\ref{tb:non-implicazioni}. 

\end{example}

\begin{example}[\LLC-local controllability$\centernot\implies$\STLC-local controllability]\label{rk:LnonST}
We now present an example of a control system in $\R^{2}$ that is \LLC-locally controllable, but which fails to be \STLC-locally controllable at the origin. 
The example shows, in particular, that a \LLC-locally controllable system is not necessarily \STLLC-locally controllable and that
a locally controllable system is not necessarily  \STLC-locally controllable,
as indicated in Table~\ref{tb:non-implicazioni}. 
Let us define, for all $x=(x_{1},x_{2})%^{\top}
\in \R^{2}$,
\[
	X_{0}(x)=\begin{pmatrix}1\\0\end{pmatrix}, \qquad X_{1}(x)=\begin{pmatrix}-x_{2}\\x_{1}\end{pmatrix}, \quad \mbox{and}\quad X_{2}(x)=\begin{pmatrix}x_{1}\\x_{2}\end{pmatrix},
\]
and let us consider the control system 
\[
	x'=u_{0}X_{0}(x)+u_{1}X_{1}(x)+u_{2}X_{2}(x), \qquad u_{0}\in [0,1],~ u_{1},u_{2}\in [-1,1].
\]
An illustration of this system can be found in Figure \ref{img:fig-L-non-ST}. Outside of the origin this system is \STLLC-locally controllable, since the vector fields $X_{1}$ and $X_{2}$ are transversal and $u_1,u_2$ can take both positive and negative values. 
One can check that the maximal angular velocity is independent of the radius, and that 
the time needed to complete a semicircle is greater than or equal to $\pi$.

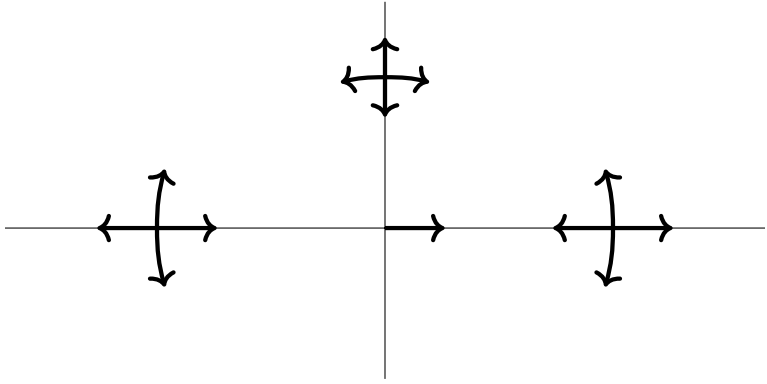
\begin{figure}
\begin{tikzpicture}
\draw
	(5,0) -- (-5, 0)
        (0,3) -- (0,-2);
\draw[->,ultra thick] 	(0,0) -- (0.785,0); %
\draw[->,ultra thick] 
	(3,0) to [out=90,in=-75] (2.898, 0.776);
\draw[->,ultra thick] 
	(3,0)	to [out=-90,in=+75] (2.898, -0.776);
\draw[->,ultra thick] 
	(3,0) -- (2.214, 0);
\draw[->,ultra thick] 
	(3,0) -- (3.785, 0); 			
\draw[->,ultra thick]				%
	(-3,0) to [out=90,in=-105] (-2.898, 0.776);
\draw[->,ultra thick] 
	(-3,0)	to [out=-90,in=105] (-2.898, -0.776);
\draw[->,ultra thick] 
	(-3,0) -- (-2.214, 0);
\draw[->,ultra thick] 
	(-3,0) -- (-3.785, 0);
\draw[->,ultra thick]				%
	(0,2) to [out=180,in=15] (-0.581, 1.932);
\draw[->,ultra thick] 
	(0,2)	to [out=0,in=165] (0.581, 1.932);
\draw[->,ultra thick] 
	(0,2) -- (0,2.523);
\draw[->,ultra thick] 
	(0,2) -- (0,1.476);
\end{tikzpicture}
\caption{\label{img:fig-L-non-ST}An illustration of the admissible vector fields of the control system in Example \ref{rk:LnonST}, which is \protect\LLC-locally controllable everywhere but is not \protect\STLC-locally controllable at the origin.}
\end{figure}
\end{example}

%%%%%%%%%%%%%%%%%%%%%%%%%%%%%%%%%%%%

	\section{Proofs}\label{sec:proofs}
	
\subsection{Preliminaries} 
Let us observe that, fixed a control  $u\in\U$, the non-autonomous differential equation \eqref{eqn:system} is well-posed in the space of %Lipschitz 
absolutely
continuous functions.
Precisely, for a given initial condition $x\in M$, there exist $T>0$ and a neighborhood $ W$ of $x$ such that $\phi(t,y,u)$ is defined for $(t,y)\in [0,T]\times  W$ and 
absolutely continuous with respect to $t$.
%Lipschitz in time. 
Moreover, for any $t\in [0,T]$ the flow $\phi(t, \cdot,u)$ restricted to $ W$ is a local diffeomorphism (see, e.g., %\cite[Thm.~I.3.1]{MR0069338}
\cite[Thm.~6.2]{jean:hal-01263668} or \cite[Thm.~1]{MR1640001}).
% Mettere qui Lip. e 

Let us denote by $\F=\{F(\cdot, u)\mid u \in U\}$ the family  of vector fields of $M$ parametrized by $F$. 
For a fixed smooth vector field $f$ in $\F$, we denote by $e^{tf}(y)$  the value at time $t$ of the trajectory  of $\dot x=f(x)$  starting from $y$, 
implicitly assuming that such a trajectory is indeed defined between $0$ and $t$.

Given a point $x$ in $M$, the \textit{controllable set to $x$} is the set of points which can be steered to $x$, i.e., 
	\[
	\A_{x}^{-} = \{y\in M \mid x\in \A_{y}\}.
	\]
Observe that  $\A^{-}_{x}$ is the attainable set  from $x$ for the control system defined by $-F$, whose solutions are the trajectories of \eqref{eqn:system} followed in %with 
the opposite time direction.

\begin{remark}\label{rmk:open}
Assume that \eqref{eqn:system} is approximately controllable. If a point $x$ in $M$ satisfies $\interior \A^{-}_{x}\neq \emptyset$, then $x$ can be reached from any other point. Indeed, for any $y$ in $M$, since $\A_{y}$ is dense in $M$ it intersects the interior of $A^{-}_{x}$.  Thus, there exists $z\in A_y\cap A_x^-$. Since $z\in \A_x$ and $z\in \A_y$, 
system \eqref{eqn:system} can be steered from $y$ to $x$  (see Figure \ref{img:fig-2}).

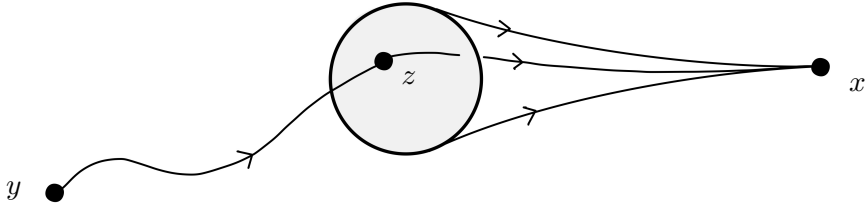
\begin{figure}
\begin{tikzpicture}
    \draw (0, 0) node[inner sep=0] {\includegraphics[width=0.8\linewidth]{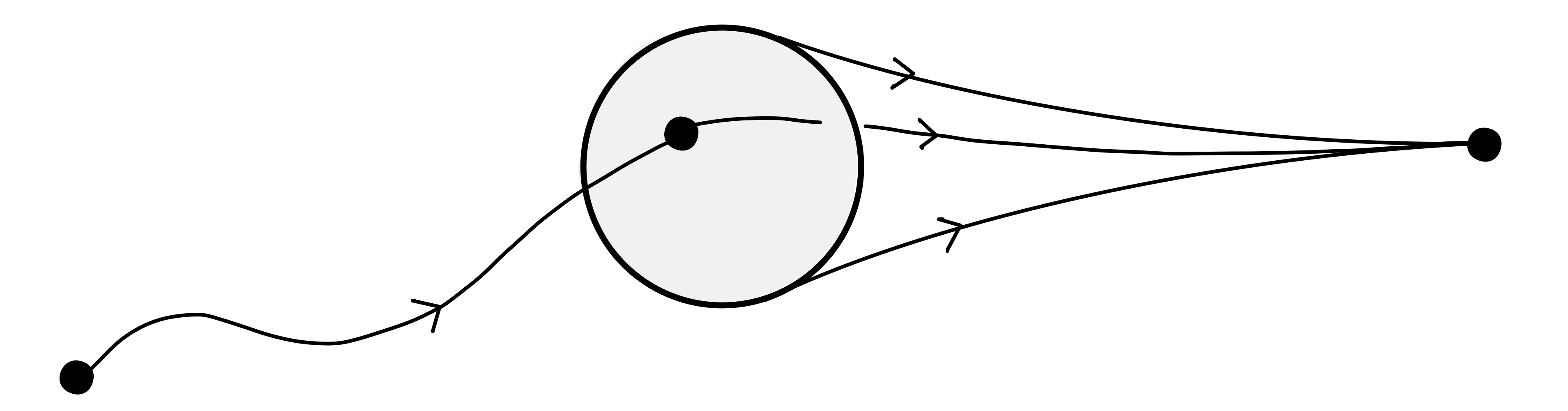}};
    \draw (-0.4, 0.3) node {$z$};
   % \draw (-0.3, -1) node {$\A_{x}^{-}$};
    \draw (-5.6,-1.2) node {$y$};
    \draw (5.5, 0.2) node {$x$};
\end{tikzpicture} 
	\caption{\label{img:fig-2} If \eqref{eqn:system} is approximately controllable and the controllable set to a certain state $x$ has nonempty interior, then $x$ can be reached from any other state in $M$. This is observed in Remark~\ref{rmk:open}.}
\end{figure}

\end{remark}

\subsection{Approximate controllability of locally controllable systems}
We are ready to prove Lemma~\ref{lem:approx}. 
\begin{proof}[Proof of Lemma~\ref{lem:approx}]
Let $x\in M$. We want to show that $\cl(\A_x)$ is open. By connectedness of $M$, this implies that $\cl(\A_x)=M$, thus proving the lemma.

Let $y\in\cl(\A_x)$. We claim that for every control $u\in \U$ and $t>0$ such that 
$\phi(t,y, u)$ is defined, 
%the flow of $F(\cdot, u)$ is defined at $t$,
we have that
\begin{equation}
    \label{eq:l2-claim}
   \phi(t,y, u)\in \cl(\A_x).
\end{equation} 
This concludes the proof of the lemma. Indeed, from \eqref{eq:l2-claim} it follows that  $\A_{y}\subset \cl(\A_x)$; since $\A_{y}$ %is open and
contains $y$ in its interior
due to  local controllability,  this proves that $\cl(\A_x)$ is open. 

In order to prove \eqref{eq:l2-claim}, fix any neighborhood $V$ of $ \phi(t,y, u)$: we show that $V$  has nonempty intersection with $\A_x$.
Consider a neighborhood $W$ of $y$ such that the map $\varphi=\phi(t,\cdot,u)|_W$ is a diffeomorphism. 
In particular, $\varphi(W)$ is a neighborhood of  $\phi(t,y, u)$, and the set $W'=\varphi^{-1}(V \cap  \varphi(W))$ is a neighborhood of ${y}$. 
Since $y$ is in the closure of $\A_x$, there exists $y_1\in W'\cap \A_x$. 
Consider an admissible control steering \eqref{eqn:system} from $x$ to $y_1$: by concatenating such a control with $u$ one finds that $\phi(t,y_{1},u)$ is in $\A_x$. 
This implies that $\phi(t,y_{1},u)$ belongs to $V\cap \A_x$, proving that $V\cap \A_x$ is nonempty, as required.
\end{proof}

\subsection{Symmetry of attainable sets of locally controllable systems} 

\begin{proof}[Proof of Lemma \ref{lem:salmone}]
Let $x$ and $y$ in $M$ be such that $y\in\A_{x}$.  We argue by contradiction supposing that $x\notin\A_{y}$. 
We claim that this implies the existence of a point $z$ in $M$ (actually $z\in\A_{x}$) such that 
\begin{equation}\label{eq:claim}
z\notin \A_{y} \quad \text{and} \quad \interior \A_{z}^{-}\neq \emptyset.
\end{equation}
%(Compare with Figure \ref{img:risali}.)
This yields a contradiction, since the assertions  in \eqref{eq:claim} cannot hold  both at the same time  due to Remark~\ref{rmk:open}.  (Notice that Remark~\ref{rmk:open} applies to system~\eqref{eqn:system} because of Lemma~\ref{lem:approx}.)
The rest of the proof is dedicated to proving the existence of a point $z$ satisfying \eqref{eq:claim}. 

Consider %a control 
$u\in\mathcal U$ and $T>0$ such that $\phi(T,x,u)=y$. 
Define the absolutely continuous curve $ \gamma:[0,T]\to M$ by $\gamma(t)=\phi(t,x,u)$. 
Let 
\[
\tau = \inf\{t \in [0,T]\mid \gamma(t)\in \A_{y}\}.
\]
	We claim that $\gamma([0,T])\cap\A_{y} = \gamma((\tau,T])$ (see Figure~\ref{img:risali}).
	Indeed, $\gamma^{-1}(\gamma([0,T])\cap\A_{y})$ is open since $\A_{y}$ is open, and its complement is nonempty since it contains zero (we are assuming that $x\notin\A_{y}$).
	Moreover,	if a certain $s\in[0,T]$ satisfies $\gamma(s)\in \A_{y}$, then, for all $t$ in $[s,T]$, one has $\gamma(t)\in \A_{y}$ since it suffices to concatenate the control 
	steering \eqref{eqn:system}
	from $y$ to  $\gamma(s)$ with $u|_{[s,t]}$ in order to attain $\gamma(t)$. 
	Up to renaming $\gamma(\tau)$ as $x$, we can assume that $\tau = 0$. Namely, without loss of generality, we can assume 
	\[
	x\notin \A_{y} \quad \mbox{and}\quad 
	\phi(t,x,u) %	e^{tf_{1}}(x)
	\in \A_{y} \mbox{ for all } t\in (0,T]. % (0,t_1].
	\]

	Let $V$ be a neighborhood of $x$ contained in $\A_{x}$.
We now construct a parametrization $C_n:I_n\to M$ ($I_n\subset\R^n$)  of a $n$-dimensional embedded submanifold of $M$ satisfying $C_n(I_n)\subset \A^{-}_{x}$, and therefore $x$ (or more exactly $\gamma(\tau)$) satisfies  \eqref{eq:claim}.
This will be done by a recursive argument, by constructing a finite sequence of embeddings  $C_k:I_k\to M$ ($I_k\subset\R^k$), $k=1,\dots,n$, with
	\begin{equation}\label{eq:iter}
	C_k(s)\in V \quad \text{and}\quad x\in \A_{C_k(s)},\quad \text{ for all } s\in I_k.
	\end{equation}

Let us begin with $k=1$. Let  $v\in U$ (constant) such that $F(x,v)\neq 0$
and denote $f_1=F(\cdot,v)$.
	 Let $I_{1}$ be an open interval of the form $(0,\delta_{1})$ such that the map $C_{1}\colon I_{1}\to M$ defined by $C_{1}(t)=e^{-tf_{1}}(x)$ parameterizes an embedded curve. %, and fix a point $x_{1}\in \Omega\cap C_{1}(I)$. 
	Since the constant control defined by $v$ belongs to $\U$, then one can reach $x$ from any point in $C_{1}(I_1)$.
	Moreover, one has that $C_1(I_1)$ is contained in $V$, up to choosing $\delta_{1}$ sufficiently small. Thus, $C_{1}$ satisfies  \eqref{eq:iter} for $k=1$.
	
	Now, suppose having constructed a $k$-dimensional parameterization $C_{k}$ satisfying~\eqref{eq:iter}, with $1\leq k\leq n-1$.
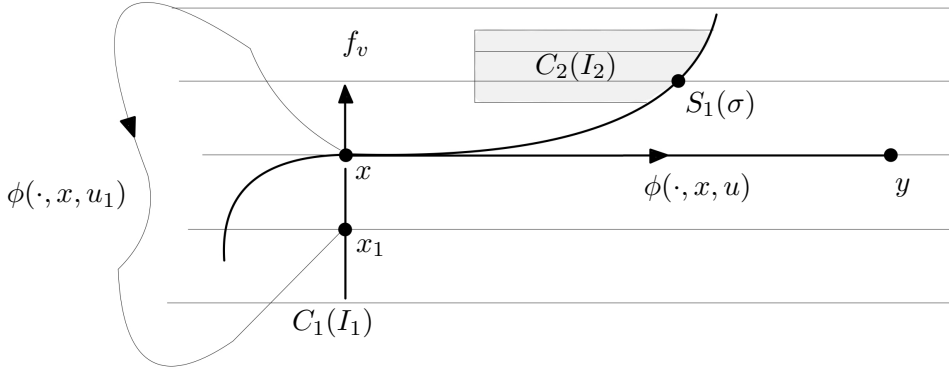
\begin{figure}
\begin{tikzpicture}
    \draw (-3, -0.05) node[inner sep=0] {\includegraphics[width=0.8\linewidth]{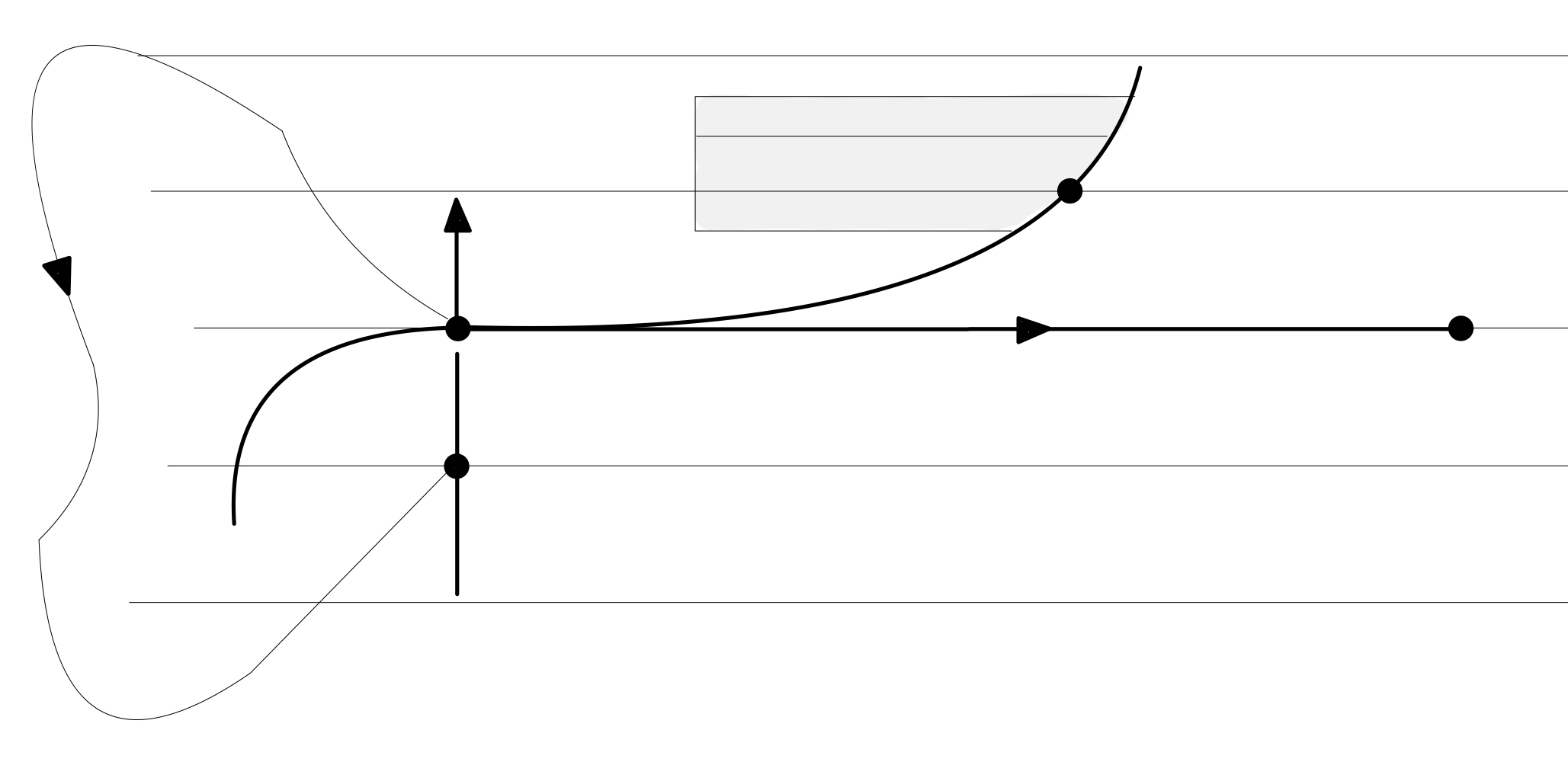}};
      \draw (5, 0.5) node {~};
    \draw (2, -0.1) node {$y$};
    \draw (-5.1, 0.1) node {$x$};
        \draw (-5, -0.9) node {$x_{1}$};
    \draw (-5.5, -1.9) node {$C_{1}(I_1)$};
    \draw (-9, -0.2) node {$\phi(\cdot,x,u_{1})$};
    \draw (-0.7, -0.1) node {$\phi(\cdot,x,u)$};
     \draw (-5.2, 1.8) node {$f_{v}$};
    \draw (-0.4, 1) node {$S_{1}(\sigma)$};
    \draw (-2.3, 1.5) node {$C_{2}(I_2)$};
\end{tikzpicture} 
	\caption{\label{img:salmone} A graphic representation of the iterations in the proof of Lemma \ref{lem:salmone}. % in dimension two, and for piecewise constants controls. 
	One can steer %control 
	any point in $C_{2}(I_2)$ to $x$ by first attaining $S_{1}(I_1)$, then attaining $C_{1}(I_1)$ via the control $u_{1}$ from which one can reach the initial state $x$.}
\end{figure}
	Fix a point $x_{k}\in C_{k}(I_{k})$, and consider a control $u_{k}$ in $\U$ and a time $T_{k}\geq0$ such that $x_{k}=\phi(T_{k},x,u_{k})$. (We are using here that $x_k\in V\subset \A_x$.) 
	Let $W_{k}$ be a neighborhood of $x$ such that $\varphi_{k}=\phi(T_{k},\cdot,u_{k})|_{W_k}$ is a diffeomorphism.
	Then $S_{k}\vcentcolon=\varphi_{k}^{-1}\circ C_{k}$ parameterizes an embedded submanifold of dimension $k$ containing $x$. 
	Moreover,
	\begin{equation}\label{eq:S_kx}
	x\in\A_{S_{k}(s)}, \qquad \forall s\in I_{k},
	\end{equation}
	since $x\in \A_{C_k(s)}$, and $C_k(s)\in \A_{S_k(s)}$ using $u_k$ as control. 
	In particular, we have that $S_k(s)\notin \A_y$ for all $s\in I_k$.
	As a consequence, since $\phi(t,x,u)\in \A_{y}$ for all $t\in (0,T]$, we have that
	\begin{equation}
	    \label{eq:intersez}
	S_{k}(I_{k})\cap \{\phi(t,x,u)\mid t\in (0,T] \} =\emptyset.
	\end{equation}
	This implies the existence of $t_{k}\in [0,T]$ and of $\sigma\in I_{k}$ with $S_k(\sigma)\in V$ %is arbitrarily close to $x$  
	such that $F(S_{k}(\sigma),u(t_{k}))$ is transverse to 
	$T_{S_{k}(\sigma)}S_{k}(I_k)$.
	Indeed, if one had $F(S_{k}(s),u(t))\in T_{S_{k}(s)}S_{k}(I_k)$ for all $s\in I_k$ and $t\in[0,T]$, then, by uniqueness of solutions, $\phi(t,x,u)$ 
	%the flow starting from $x$ 
	would stay in %belong to 
	$S_k(I_k)$, at least for $t$ sufficiently small. However, this contradicts~\eqref{eq:intersez}. Moreover, $\sigma$ can be chosen so that $S_k(\sigma)$ is arbitrarily close to $x$, and in particular so that it belongs to $V$. Let $f_{k}$ be the vector field ${f_k=}F(\cdot,u(t_{k}))$.
	
	By transversality of ${f_k(S_k(\sigma))}$ and
	$T_{S_{k}(\sigma)}S_{k}(I_k)$,
	%This implies that
	 there exist $\delta_{k+1}>0$ and an open neighborhood  $I_{k}'\subset I_{k}$ containing $\sigma$  such that the map $C_{k+1}:%I_{k+1}\vcentcolon=
	 I_{k}'\times (-\delta_{k+1},\delta_{k+1})\to M$ defined by
	\[
	C_{k+1}(s,t)=e^{tf_{k}}\circ S_{k}(s), \qquad \forall (s,t)\in I_{k}'\times  %(0,\delta_{k+1}),
	{
	(-\delta_{k+1},\delta_{k+1})}, 
	\]
	is a parametrization of an embedded submanifold of dimension $k+1$.
	Moreover, since $C_{k+1}(\sigma,0)=S_k(\sigma)\in V$, the set $
	 I_{k}'\times (-\delta_{k+1},\delta_{k+1})$ can be chosen so that $C_{k+1}(I_{k}'\times (-\delta_{k+1},\delta_{k+1}))\subset V$. 
	We are now left to observe that $x\in \A_{C_{k+1}(s)}$ for all 
	$s\in{I_{k+1}\coloneqq} I_k'\times(-\delta_{k+1},0)$.
	In fact, starting from $C_{k+1}(s)$ one can reach $S_{k}(I_k)$ using the (constant) control $u(t_k)$ corresponding to $f_{k}$, and $x$ can be reached from $S_{k}(I_k)$ by \eqref{eq:S_kx}. %assumption. 
	This concludes the iteration, since $C_{k+1}{|_{I_{k+1}}}$ satisfies \eqref{eq:iter}.
% 	\hfill$\square$\\
\end{proof}
	
\subsection{Conclusion}

Once Lemma \ref{lem:salmone} is proven, Theorem \ref{thm:thm-1} follows from the following %classical 
standard argument.

\begin{proof}[Proof of Theorem \ref{thm:thm-1}]
Assume that \eqref{eqn:system} is locally controllable. Define the relation $\sim$ on $M$ by saying that $x\sim y$ if and only if $x\in \A_{y}$. 
Thanks to Lemma~\ref{lem:salmone}, $\sim$ is an equivalence relation.
Due to the local controllability, the equivalence classes are open. Each class is also closed, since its complement is the union of the other classes, and such an union is open. 
Due to the connectedness of $M$, there is only one class and system \eqref{eqn:system} is controllable.
\end{proof}

\appendix
\section{A simpler proof %of controllability 
	when \protect\LLC-local controllability holds.}
\label{app}

In this {appendix we give} a simpler proof of the  controllability of \LLC-locally controllable systems. 
We first observe the following. 
\begin{proposition}\label{prop:LLC}
If system  \eqref{eqn:system} is \LLC-locally controllable, then for any point $x\in M$ the set $\A_{x}^{-}$ has nonempty interior.
\end{proposition}

A proof of this proposition can be found below. 
However, observe that  Proposition~\ref{prop:LLC} can be directly deduced from \cite[Theorem 5.3]{Grasse1992}, since the property of localized local controllability implies, in the terminology of \cite{Grasse1992}, that 
\eqref{eqn:system} has the 
{nontangency property}.

\begin{proof}
The argument mimics the proof of Krener's theorem \cite{krener1974generalization}.
Fix $x$ in $M$. We claim that there exists  $f_{1}\in \F$ such that $f_{1}(x)\neq 0$. 
Indeed, if that were not the case, any solution $\phi(\cdot, x, u)$ with $u\in \U$ would be constant. 
Let 
\[
N_{1}=\{e^{-tf_{1}}(x)\mid t\in (0,\delta)\}
\]
for $\delta>0$.
If $M$ is one-dimensional, then we have concluded.
Otherwise, we claim that there exist $y_1\in N_1$ %$\delta_{1}\in (0,\delta)$ 
and $f_{2}\in \F$  such that $f_{1}(y_1)%(e^{-\delta_{1}f_{1}}(x))
$ and $f_{2}(y_1)
$ are linearly independent.
Indeed, let $V_{1}$ be a neighborhood of $e^{-\frac{\delta}{2} f_{1}}(x)$
not containing $x$ nor $e^{-\delta f_{1}}(x)$
and assume that every $f\in \F$ is tangent to $N_1\cap V_1$.  Then the trajectories of \eqref{eqn:system} starting from $N_1\cap V_1$
and staying in $V_1$ cannot quit $N_1\cap V_1$. This contradicts the localized local controllability property.

Thus, define the embedded two-dimensional submanifold
	\[
	N_{2} =\{e^{-t_{2}f_{2}}\circ e^{-t_1f_{1}}(x)\mid (t_{1},t_{2})\in I_2\times (0,\delta_2)\},
	\]
	for a suitable nonempty open subinterval $I_2$ of $(0,\delta)$ and a suitable
	$\delta_2>0$.
%	open neighborhood $I_{2}$ of $\delta_1$.
If the dimension of $M$ is equal to $2$ the proof is concluded, otherwise, reasoning as above, 
there exist $y_2=\in N_2$ and $f_3\in\F$ such that $f_3(y_2)$ is transverse to $N_2$, 	  i.e., 
	\begin{equation*} %\label{eq:trans}
	\dim (\spa\{f_3(y_2)\}+ T_{y_2}N_{2}) =\dim(N_2)+1=3.
\end{equation*}
Hence, the differential of the map $(t_1,t_2,t_3)\mapsto e^{-t_{3}f_{3}}\circ e^{-t_{2}f_{2}}\circ e^{-t_1f_{1}}(x)$ has full rank in a neighborhood of $(\bar t_1,\bar t_2,0)$, where $\bar t_1,\bar t_2$ are such that $y_2=e^{-\bar t_{2}f_{2}}\circ e^{-t\bar f_{1}}(x)$. We can then iterate the construction up to reaching the dimension of $M$. 
\end{proof}

{
Now due to Lemma~\ref{lem:approx}, \LLC-locally controllable systems are approximately controllable. Moreover, Proposition~\ref{prop:LLC} ensures that any point $x$ in $M$ satisfies $\interior \A^{-}_{x}\neq \emptyset$.
Thus, by Remark~\ref{rmk:open}, we deduce that any $x$ in $M$ can be reached from any other point.
}

\bibliographystyle{alphaabbrv}
\bibliography{biblio-v-UU}

% \bib, bibdiv, biblist are defined by the amsrefs package.
\begin{bibdiv}
\begin{biblist}

\bib{bressan}{book}{
      author={Bressan, Alberto},
      author={Piccoli, Benedetto},
       title={Introduction to the mathematical theory of control},
      series={AIMS Series on Applied Mathematics},
   publisher={American Institute of Mathematical Sciences (AIMS), Springfield,
  MO},
        date={2007},
      volume={2},
        ISBN={978-1-60133-002-4; 1-60133-002-2},
      review={\MR{MR2347697 (2008j:49001)}},
}

\bib{MR689492}{article}{
      author={Bacciotti, Andrea},
      author={Stefani, Gianna},
       title={On the relationship between global and local controllability},
        date={1983},
        ISSN={0025-5661},
     journal={Math. Systems Theory},
      volume={16},
      number={1},
       pages={79\ndash 91},
         url={https://doi.org/10.1007/BF01744571},
      review={\MR{689492}},
}

\bib{parsons2007}{book}{
      author={Choset, Howie},
      author={Lynch, Kevin~M.},
      author={Hutchinson, Seth},
      author={Kantor, George},
      author={Burgard, Wolfram},
      author={Kavraki, Lydia~E.},
      author={Thrun, Sebastian},
       title={Principles of robot motion: Theory, algorithms and
  implementations},
   publisher={Cambridge University Press},
        date={2007},
      volume={22},
      number={2},
}

\bib{Coron}{book}{
      author={Coron, Jean-Michel},
       title={Control and nonlinearity},
      series={Mathematical Surveys and Monographs},
   publisher={American Mathematical Society, Providence, RI},
        date={2007},
      volume={136},
        ISBN={978-0-8218-3668-2; 0-8218-3668-4},
         url={https://doi.org/10.1090/surv/136},
      review={\MR{2302744}},
}

\bib{MR774166}{article}{
      author={Grasse, Kevin~A.},
       title={A condition equivalent to global controllability in systems of
  vector fields},
        date={1985},
        ISSN={0022-0396},
     journal={J. Differential Equations},
      volume={56},
      number={2},
       pages={263\ndash 269},
         url={https://doi.org/10.1016/0022-0396(85)90108-1},
      review={\MR{774166}},
}

\bib{Grasse1992}{article}{
      author={Grasse, Kevin~A.},
       title={On the relation between small-time local controllability and
  normal self-reachability},
        date={1992},
        ISSN={0932-4194},
     journal={Math. Control Signals Systems},
      volume={5},
      number={1},
       pages={41\ndash 66},
         url={https://doi.org/10.1007/BF01211975},
      review={\MR{1137191}},
}

\bib{MR1346883}{article}{
      author={Grasse, Kevin~A.},
       title={Reachability of interior states by piecewise constant controls},
        date={1995},
        ISSN={0933-7741},
     journal={Forum Math.},
      volume={7},
      number={5},
       pages={607\ndash 628},
         url={https://doi.org/10.1515/form.1995.7.607},
      review={\MR{1346883}},
}

\bib{jean:hal-01263668}{book}{
      author={Jean, Fr{\'e}d{\'e}ric},
       title={{Stabilit{\'e} et commande des syst{\`e}mes dynamiques}},
     edition={2},
   publisher={{Les} presses de l'Ecole Nationale Sup{\'e}rieure de Techniques
  Avanc{\'e}es (ENSTA)},
        date={2017},
         url={https://hal-ensta-paris.archives-ouvertes.fr//hal-01263668},
}

\bib{Krastanov2021}{article}{
      author={Krastanov, Mikhail~Ivanov},
      author={Nikolova, Margarita~Nikolaeva},
       title={A necessary condition for small-time local controllability},
        date={2021},
        ISSN={0005-1098},
     journal={Automatica J. IFAC},
      volume={124},
       pages={109258, 5},
         url={https://doi.org/10.1016/j.automatica.2020.109258},
      review={\MR{4192998}},
}

\bib{krener1974generalization}{article}{
      author={Krener, Arthur~J},
       title={A generalization of {C}how's theorem and the bang-bang theorem to
  nonlinear control problems},
        date={1974},
     journal={SIAM Journal on Control},
      volume={12},
      number={1},
       pages={43\ndash 52},
}

\bib{MR692840}{article}{
      author={Kupka, I.},
      author={Sallet, G.},
       title={A sufficient condition for the transitivity of pseudosemigroups:
  application to system theory},
        date={1983},
        ISSN={0022-0396},
     journal={J. Differential Equations},
      volume={47},
      number={3},
       pages={462\ndash 470},
         url={https://doi.org/10.1016/0022-0396(83)90045-1},
      review={\MR{692840}},
}

\bib{MR1640001}{book}{
      author={Sontag, Eduardo~D.},
       title={Mathematical control theory},
     edition={Second},
      series={Texts in Applied Mathematics},
   publisher={Springer-Verlag, New York},
        date={1998},
      volume={6},
        ISBN={0-387-98489-5},
         url={https://doi.org/10.1007/978-1-4612-0577-7},
        note={Deterministic finite-dimensional systems},
      review={\MR{1640001}},
}

\bib{MR872457}{article}{
      author={Sussmann, H.~J.},
       title={A general theorem on local controllability},
        date={1987},
        ISSN={0363-0129},
     journal={SIAM J. Control Optim.},
      volume={25},
      number={1},
       pages={158\ndash 194},
         url={https://doi.org/10.1137/0325011},
      review={\MR{872457}},
}

\end{biblist}
\end{bibdiv}

\end{document}